\documentclass{amsart}

\usepackage{hyperref,color}

\usepackage{amsmath, amssymb, latexsym, amsfonts, amscd, amsthm,  amstext}
\newtheorem{theorem}{Theorem}[section]

\newtheorem{corollary}[theorem]{Corollary}

\newtheorem{lemma}[theorem]{Lemma}

\newcommand{\bh}{B_s(\mathcal{H} )}

\newcommand{\R}{\mathbb{R}}

\newcommand{\qcom}{\leftrightarrow_q}
\newcommand{\com}{\leftrightarrow_c}
\newcommand{\jcom}{\leftrightarrow_J}
\begin{document}
	\title[]{Maps on self-adjoint operators preserving some relations related to commutativity}
	
	\author[]{Mahdi Karder}
	\address{}
	\email{mahdi.karder@uoz.ac.ir}
	
	\author[]{Tatjana Petek}
	\email{tatjana.petek@um.si}
	
	
	\date{\today}

	
	\begin{abstract}
		In this paper, we give the complete description of maps on self-adjoint bounded operators on Hilbert space which preserve a triadic relation involving the difference of operators and either commutativity or quasi-commutativity in both directions.  We show that those maps are implemented by unitary or antiunitary equivalence and possible additive perturbation by a scalar operator.
	\end{abstract}
	\maketitle
	\section{Introduction and statement of the results}
	Let $\mathcal{H}$ be a complex Hilbert space with $\dim \mathcal{H} \geq 3$. The symbols $B(\mathcal{H})$ and $\bh$ will denote the set of all bounded linear operators and all bounded self-adjoint operators acting on $\mathcal{H}$, respectively. Two bounded operators $A$ and $ B$  are said to quasi-commute if there exists a nonzero $\lambda\in \mathbb{C}$ such that $AB = \lambda BA$. In \cite{brooke}, one finds the following surprising result: self-adjoint operators $A$ and $B$ quasi-commute if and only if they commute $(AB = BA)$ or anti-commute $(AB = -BA)$. 
	\par The structure of bijective transformations that preserve commutativity or quasi-commutativity in both directions was obtained by L. Moln{\'a}r and P. {\v{S}}emrl in \cite{molnar1} and G. Dolinar and B. Kuzma in \cite{dolinar3},\cite{dolinar1}.  We state their results below.
	
	\begin{theorem}\cite{molnar1}\label{th:1}
		Let $\mathcal{H}$ be a complex separable Hilbert space with $\dim \mathcal{H} \geq 3$ and let $ \phi:\bh \mapsto \bh  $	be a bijective transformation which preserves commutativity in both directions. Then there exists either a unitary or an antiunitary operator $U$ on $\mathcal{H}$ and for every operator
		$A \in \bh$ there is a real-valued bounded Borel function $f_A$ on $ \sigma (A)$ such that
		\[ \phi(A) = Uf_A(A) U^{\ast}\quad (A\in\bh), \]
		where $\sigma(A)$ denotes the spectrum of $A$.
	\end{theorem}
	
	\begin{theorem}\cite{dolinar1}\label{th:3}
		Let $\mathcal{H}$ be a complex separable Hilbert space with $\dim \mathcal{H} \geq 3$ and let $ \phi:\bh \mapsto \bh  $	be a bijective transformation which preserves quasi-commutativity in both directions. Then there exists either a unitary or an antiunitary operator $U$ on $\mathcal{H}$ and for every operator
		$A \in \bh$ there is a real-valued bounded Borel function $f_ A$ on $\sigma(A)$   such that
		\[ \phi(A) = Uf_A(A) U^{\ast}\quad (A\in\bh). \]
	\end{theorem}
	The natural question arises here: Can we add some weak conditions on our transformation to preserve commutativity or quasi-commutativity in both directions to get more regular forms for such a transformation? For some works in this direction, we refer to \cite{geher}, \cite{molnar2}, \cite{nagy}. One possible way might be only to strengthen the conditions in Theorem \ref{th:1} and \ref{th:3}, respectively. This works and is not completely trivial; however, we decided to use different assumptions at the start so that we cannot use Theorems $\ref{th:1}$ and $\ref{th:3}$ directly.  Let us underline that we do not assume separability of the underlying Hilbert space and therefore, cannot use strong analytic tools. Moreover, it turned out that we do not need to assume injectivity. Of course, some additional assumptions must be in place. The strengthening of commutativity and quasi-commutativity conditions provided positive results.   
	
	Before we state our main results, let us fix some notations. In the following definitions, let $A$, $B\in\bh$ and $\{\} \neq \mathcal{M}\subseteq \bh$.
	\begin{align*}
		&(1)~A\circ B=AB+BA\\
		&(2)~A \com B \Longleftrightarrow AB-BA=0\\
		&(3)~A\jcom B \Longleftrightarrow A\circ B=0\\
		&(4)~A \qcom B \Longleftrightarrow A \com B \text{ or } A \jcom B\\
		&(5)~\mathcal{M}^{c}=\{X\in \bh ~|~ X \com M,\, M\in \mathcal{M}\}	\\	
		&(6)~\mathcal{M}^{\#}=\{X\in \bh ~|~ X \qcom M,\, M\in \mathcal{M}\}	\\	
		&(7)~\mathcal{M}^{cc}=(\mathcal{M}^{c})^c \\
		&(8)~\mathcal{M}^{\#\#}=(\mathcal{M}^{\#})^\# \\
		&(9)~A^{c}=\{A\}^{c}	\\	
		&(10)~A^{cc}=\{A\}^{cc} \\ 
		&(11)~A^{\#}=\{A\}^{\#}	\\
		&(12)~ A^{\#\#}=\{A\}^{\#\#}. 
	\end{align*}
	Theorems \ref{th:4} and  \ref{th:5} are our main results.
	\begin{theorem}\label{th:4}
		Let $\mathcal{H}$ be a  complex  Hilbert space with $\dim \mathcal{H} \geq 3$ and let $ \phi:\bh \mapsto \bh  $	be a surjective mapping preserving the following triadic relation (related to commutativity)
		\begin{equation}\label{eq:0}
			A-B\com C \Longleftrightarrow \phi(A)-\phi(B)\com \phi(C),\ \ \ (A,B,C\in \bh),
		\end{equation}
		Then there exist a nonzero real scalar $c$, either a unitary or an antiunitary operator $U$ on $\mathcal{H}$ and a function $g: \bh \rightarrow \mathbb{R} $ such that
		\begin{equation} \label{eq:00}
			\phi(A) = c U A U^{\ast} +g(A)I.
		\end{equation}
		for all $A\in\bh$.
	\end{theorem}
	It is easy to see that the application of the previous theorem gives rise to the following corollary.
	\begin{corollary}
		Suppose $\phi:\bh \to \bh$ is an additive surjective map preserving commutativity in both directions. Then there exist a nonzero real scalar $c$, either a unitary or an antiunitary operator $U$ on $\mathcal{H}$ and an additive function $g: \bh \rightarrow \mathbb{R} $ such that $\phi$ is of the form \eqref{eq:00}.
	\end{corollary}
	\begin{theorem}\label{th:5}
		Let $\mathcal{H}$ be a   complex  Hilbert space with $\dim \mathcal{H} \geq 3$ and let $ \phi:\bh \mapsto \bh  $	be a surjective transformation preserving the following triadic relation (related to quasi-commutativity)
		\begin{equation}\label{eq:1}
			A-B\qcom C \Longleftrightarrow \phi(A)-\phi(B)\qcom \phi(C),\ \ \ (A,B,C\in \bh).
		\end{equation}
		Then there exist a nonzero real scalar $c$, a real valued function $f$ on $\bh$ and either a unitary or an antiunitary operator $U$ on $\mathcal{H}$  such that
		\begin{equation}\label{eq:33}
			\phi(A) = c U A U^{\ast} + f(A)I.
		\end{equation}
		The function $f$ necessarily vanishes on the set of all operators $A$ having the property that there exists a nonzero operator $B\in \bh$, not commuting with $A$, such that $AB+BA=0$. 
	\end{theorem}
	Similarly to the above, we state a direct consequence of the above theorem.
	\begin{corollary}
		Suppose $\phi:\bh \to \bh$ is an additive surjective map preserving quasi-commutativity in both directions. Then there exist a nonzero real scalar $c$, either a unitary or an antiunitary operator $U$ on $\mathcal{H}$ and an additive function $f: \bh \rightarrow \mathbb{R} $ such that $\phi$ is of the form \eqref{eq:33}.
	\end{corollary}
	
	Before we start with proofs, we state some technical Lemmas which will be used in both proofs.
	
	\begin{lemma}\label{lem:scalar}
		Let $A\in \bh$. The following assertions are equivalent.
		\begin{enumerate}
			\item $A\in \R I$.
			\item $A^{c}=\bh$.
			\item $A^{\#}=\bh$.
			\item $B-A\com B$ for every $B\in \bh$.
			\item $B-A\qcom B$ for every $B\in \bh$.
		\end{enumerate}
	\end{lemma}
	\begin{proof}
		The implications (1) $\Rightarrow$(2)$\Rightarrow$(3)$\Rightarrow$(4)$\Rightarrow$(5) are obvious.

		So suppose that (5) holds true and that $A$ is not a scalar operator. We can pick a nonzero operator $T$ such that $T$ does not commute with $A$ and, in turn, $tT-A\notin T^c$ for every nonzero $t\in \R$. It follows that  $(tT-A)\circ tT=0$ for every $t\neq 0$ and so, $T^2=0=T$, a contradiction. This proves the implication  (5)  $\Rightarrow$ (1). 
	\end{proof} 

	\begin{lemma}\label{lem:4}
		If $A-\lambda I \jcom B$ and $B-\lambda I \jcom A$ for some $\lambda\neq 0$, then $A=B$.
	\end{lemma}
	\begin{proof}
		From the assumptions we compute that $A\circ B = 2\lambda A=2\lambda B$, giving that $A=B$.
	\end{proof}
	
	\begin{lemma}\label{lem:AEF} Let $H_0$ be any two-dimensional subspace of $H$ and let $H=H_0\oplus H_0^\perp$.
		For any nonzero real number $a$, let $\lambda =\sqrt{1+a^2}$, and according to the decomposition above, set
		\begin{align*}
			&A= \left(
			\begin{array}{cc}
				-a & 1  \\
				1   & a    
			\end{array} \right) \oplus \lambda I\\
			&E= \left(
			\begin{array}{cc}
				-a &  0 \\
				0   & a    
			\end{array} \right) \oplus aI\\
			&F= \left(
			\begin{array}{cc}
				0 & 1  \\
				1   & 0    
			\end{array} \right) \oplus I.
		\end{align*}
		Then $\sigma(A)=\{ \lambda, -\lambda \}$, $\sigma(E)=\{a,-a\}$ and $\sigma(F)=\{1,-1\}$. We have also   $A= \lambda (I-2P)$, $E=  a(I-2Q)$ and $F =I-2R$, where $P,Q,R$ are uniquely given rank-one projections. Moreover, for any nonzero real numbers $\alpha$, $\varepsilon$ and $\varphi$ we have $\alpha A-\varepsilon E\com F$ if and only if $\alpha = \varepsilon$, $\alpha A-\varphi F\com E$ if and only if $\alpha=\varphi$,  $(\alpha A-\varepsilon E)\circ F\ne 0$ and  $(\alpha A-\varphi F)\circ E\ne 0$. 
	\end{lemma}
	\begin{proof}
		The proof is an elementary exercise and is left to the reader.
	\end{proof}
	
	The next two Lemmas will give us a tool to characterize operators with exactly two points in the spectrum by considered relations. 
	\begin{lemma}\label{lem:1.8}
		For a nonscalar operator  $A \in\bh$, the following are equivalent.
		\begin{enumerate}
			\item $A$ has exactly two points in the spectrum.
			\item If $B^{cc} \subsetneq A^{cc}$, then  $B$ is scalar.
		\end{enumerate}
	\end{lemma}
	\begin{proof}
		That (1) implies (2) is straightforward. Assume that $A$ has more than two points in the spectrum, say $\lambda_1<\lambda_2<\lambda_3$ are distinct spectral points. Then the spectral projection $P=\chi_{\left( -\infty,\lambda_2\right] }(A)$ is nontrivial and commutes with $A$. Clearly, $A^{c}\subset P^{c}$, giving  $P^{cc}\subsetneq A^{cc}$. Now $B=P$ is nonscalar and $B^{cc}$ a proper subspace of $A^{cc}$.
	\end{proof}
	\begin{lemma}\label{lem:7}
		For every $A\in \bh$, $A^{\#\#} \subseteq A^{cc}$. On the other hand, if $A^\#=A^c$, then  $A^{\#\#} = A^{cc}$.
	\end{lemma}
	\begin{proof}
		Let $Z\in A^{\#\#}$. Clearly $A^c\subset A^{\#}$ so, in particular, $Z\leftrightarrow_q X$ for every $X\in A^c$. We need to show that $ZX=XZ $ for each $X\in A^c$, whence it follows that $Z\in A^{cc}$, as claimed. Now, assume on the contrary that $ZX=-XZ$ for some nonzero $X\in A^c$. For every real scalar $\lambda$ we have $(\lambda I+X)\in A^c \subset A^{\#} $. If there exists a real scalar $\lambda$ such that $Z(\lambda I+X)=(\lambda I+X)Z$, then $ZX=XZ$.  Otherwise,  $Z(\lambda I+X)=-(\lambda I+X)Z$ for each real scalar $\lambda$, implying that $2\lambda Z=0$, a contradiction. Then $Z$ commutes with every $X\in A^c$ thus $Z\in A^{cc}$.
		
		For the second assertion, apply that $A^{cc} \subseteq (A^c)^\# =A^{\#\#} \subseteq A^{cc}$.
	\end{proof}
	\begin{lemma}\label{lem:1.81}
		For a nonscalar operator  $A \in\bh$ such that $A^\# = A^c$, the following are equivalent.
		\begin{enumerate}
			\item $A$ has exactly two points in the spectrum, which do not add up to zero.
			\item If $B^\# = B^c$ and $B^{\#\#} \subsetneq A^{\#\#}$, then  $B$ is a scalar operator.
		\end{enumerate}
	\end{lemma}
	\begin{proof}
		Just apply Lemma \ref{lem:7}, the fact that for an operator with two points spectrum $\{\lambda,-\lambda\}$ the commutant and the quasi-commutant do not coincide, then the proof is similar to the proof of Lemma \ref{lem:1.8}.
	\end{proof}

	\begin{lemma}\label{lem:primitive}
		Let $\dim H \geq 4$ and $A=\alpha P + \beta I$ for some nontrivial projection $P$. The following are equivalent.
		\begin{enumerate}
			\item $\mathrm{rank}P>1$ and  $\mathrm{corank}P>1$.
			\item There exists an operator $B\in A^c$ with two spectral points and an operator $C$ such that $B^c\ne A^c$ and $A^{cc} \subsetneq C^{cc} \subsetneq (A-B)^{cc}$.
		\end{enumerate}
	\end{lemma}
	\begin{proof}
		Supposing $\mathrm{rank\,}P$ or  $\mathrm{corank\,}P$ is equal to one, we do not lose any generality by assuming $P$ is a rank-one projection.  If $B\in A^c$ is any operator with two spectral points, then we can write $B=\gamma Q+ \delta I$, $\gamma \neq 0$, $Q\notin \{P, I-P\} $ and $Q$ commutes with $P$. It suffices to check that $\dim (A-B)^{cc} \leq  3$. Apply that $A-B\equiv \alpha P - \gamma Q$ and thus, it can have at most three spectral points. Then there is no proper subspace between the two-dimensional $A^{cc}$ and at most three-dimensional $(A-B)^{cc}$.
		
		To prove the reversed implication, assume that both $\mathrm{rank\,}P$ and $\mathrm{corank\,}P$ are greater than or equal to $2$.   We choose a rank-one projection $Q_1$ such that $Q_1\leq P$ and a rank-one projection $Q_2$ such that $Q_2\leq I-P$.  Then the four orthogonal projections $P-Q_1, Q_1, Q_2, I-P-Q_2$ are a resolution of the identity.

		Let $B=-Q_1-Q_2-2I$ and note that $B$ has two spectral points. Operator $P-B=4Q_1+Q_2+3(P-Q_1)+2(I-P-Q_2)$ has four spectral points and its second commutant is spanned by the above-mentioned projections. Now choose say, $C=(Q_1+Q_2)+2(P-Q_1)+3(I-P-Q_2)$ and observe that $(P-B)^c \subsetneq C^c$, so $C^{cc} \subsetneq (P-B)^{cc}$. Clearly $A^{cc} \subsetneq C^{cc}$, as the dimension of $A^{cc}$ equals two and $C^{cc}$ is three-dimensional.
	\end{proof}
	\begin{lemma}\label{lem:primitive1}
		Let $\dim H \geq 4$ and $A=\alpha P + \beta I$  for some nontrivial projection $P$ such that $A^\#=A^c$. The following are equivalent.
		\begin{enumerate}
			\item $\mathrm{rank}P>1$ and  $\mathrm{corank}P>1$.
			\item There exists an operator $B\in A^\#$ such that $B^\#=B^c$ and $B^c\ne A^c$ and there exists an operator $C$ satisfying $C^\#=C^c$ such that $A^{\#\#} \subsetneq C^{\#\#} \subsetneq (A-B)^{\#\#}$.
		\end{enumerate}
	\end{lemma}
	\begin{proof}
		To follow the proof of the previous Lemma, one has to check only if the operators $B$ and $C$ meet additional requirements and this is in fact the case.
	\end{proof}
	\section{Proof of Theorem \ref{th:4}}
	Let $ \phi:\bh \mapsto \bh  $	be a surjective transformation satisfying \eqref{eq:0}. We introduce an equivalent relation on $\bh$ as follows: $B\equiv A$ if and only if $B-A\in\R I$.  By Lemma \ref{lem:scalar} we get that $A\in\R I$ if and only if $\phi(A)\in\R I$. This also quickly implies that $A\com B$ if and only if $\phi(A) \com \phi(B)$. 
	
	We assert that for any operators $A$ and $B$ in $\bh$ we have 	$(\phi(A)-\phi(B))^c=\phi(A-B)^c$ and in turn also
	\begin{equation}\label{eq:A-B}
		(\phi(A)-\phi(B))^{cc}=\phi(A-B)^{cc} 
	\end{equation}	 
	which is a straightforward consequence of the preserving property \eqref{eq:0}, surjectivity of $\phi$ and the  fact that $\phi$ preserves commutativity in both directions.
	
	As a consequence, we observe that $\phi$ is "almost" injective, that is,   $A\equiv B $ if and only if $\phi(A)\equiv\phi(B)$.

	We cannot directly apply Theorem \ref{th:1} because $H$ is not necessarily separable Hilbert space and $\phi$ is not assumed injective. But non-injectivity is not a big obstacle, for we have just shown that $\phi$ is almost injective.
	
	Further we observe that $\phi(\mathcal{M}^{cc})= \phi(\mathcal{M})^{cc}$ holds for every subset  $\mathcal{M}\subset \bh$ and the preimage of the set $\phi(\mathcal{M}^{cc})$ is equal to $\phi^{-1}(\phi(\mathcal{M}^{cc}))=\mathcal{M}^{cc}$. The verification of this statement is left to the reader. Recall that in general $\mathcal{S}\subset \phi^{-1}(\phi(\mathcal{S}))$ for a subset  $\mathcal{S}$ in the domain of $\phi$.

	Applying Lemma \ref{lem:1.8} and that $\phi(A)^{cc}=\phi(A^{cc})$, we observe that any operator has two points in the spectrum if and only if $\phi(A) $ does. 
	From \cite{molnar1} we borrow the term that an operator $A\in \bh$ is called \textit{primitive} if it is of the form $A=\alpha P +\beta I\equiv \alpha P$, $\alpha \neq0$, and $P$ is a rank-one projection. If $\dim H=3$, every operator with two spectral points is primitive. Assume further that $\dim H \ge 4 $ and that $A$ is primitive. Then by Lemma \ref{lem:primitive} we have that $A$ is primitive if and only if $\phi(A)$ is primitive. Following the consideration in  \cite{molnar1}, page 529,  we define a  map $\psi$ on rank-one projections: $\psi(R)=S$,  when $\phi(R)\equiv cS$, $c\neq 0$ and $S$ being a rank-one projection. Note that  $\psi$ is well defined since $\dim H \ge 3$. The surjectivity can be checked in the same way as in \cite{molnar1}. For the injectivity, we have to be a bit careful as $\phi$ is not assumed injective. So, suppose that $Q=\psi(P_1)=\psi(P_2)$ for two rank- one projections $P_1$ and $P_2$. Then $\phi(P_j)=\alpha_j Q +\beta_j I$, $j=1,2$. This implies that 
	$$\phi(P_1^{cc})=\phi(P_1)^{cc}=Q^{cc}=\phi(P_2)^{cc}=\phi(P_2^{cc}).$$
	Considering the preimages of $\phi(P_j^{cc})$, $j=1,2$ and using that $\phi^{-1}(\phi(P_j^{cc})=P_j^{cc}$, $j=1,2$, we easily find out that $P_1=P_2$ and thus $\psi$ injective.  So, $\psi$ is a bijective map on rank-one projections, preserving commutativity in both directions and in turn, $\psi$ preserves orthogonality in both directions.  Then, applying the well-known Uhlhorn's generalization \cite{uhl} of Wigner's theorem on quantum mechanical symmetry transformation, there exists a unitary or antiunitary operator $U$ on $\mathcal{H}$ such that 
	$$\psi(R)=URU^\ast$$
	for every rank-one projection $R$. We further suppose, without losing any generality, that $\psi(R)=R$ for every rank-one projection $R$. Then for every rank-one projection $R$ we have  $\phi(\alpha R + \beta I) \in \phi(R^{cc})=R^{cc}$ and so, $\phi(\alpha R)\equiv f_R(\alpha)R$ for some function $f_R$ with only nonzero real values.  Since $\phi(\alpha R + \beta I)-\phi(\alpha R)$ is scalar for every $\beta \in \R$, we have  
	
	$$\phi(\alpha R + \beta I)\equiv f_R(\alpha) R$$
	for every $\beta \in \R$ and every rank-one projection  $R$.
	
	We proceed in several steps.
	
	{\sc Step 1.} 
	\textit{There exists a  function $f:\R^\ast \to \R^\ast$ such that  for any nontrivial  projection $P$ and $\alpha$, $\beta \in \mathbb{R}$, $\alpha\neq 0$,}  
	\begin{equation}
		\phi(\alpha P +\beta I )\equiv f(\alpha)P. 
	\end{equation}

	Suppose that $P\perp Q$ are rank-one projections. Then $\phi(\alpha(P+Q))$ has two points in the spectrum and  from
	$$\phi(\alpha(P+Q))-\phi(\alpha P)\in (\phi(\alpha(P+Q)-\phi(\alpha P))^{cc}=\phi(\alpha Q)^{cc}=Q^{cc}$$
	we get that $\phi(\alpha(P+Q))\equiv g_{P+Q}(\alpha)(P+Q)$ for some function $g_{P+Q}$ satisfying $f_{P+Q}=f_P=f_Q$.
	If $P$ and $Q$ are not necessarily orthogonal rank-one projections, then there exists a rank-one projection  $R\perp P$, $R\perp Q$ such that $f_R=g_{R+P}=f_P$ and $f_R=g_{R+Q}=f_Q$. So we can resume that there is a function $f$, independent of $P$ and $\beta$ such that $\phi(\alpha P +\beta I )\equiv f(\alpha)P$ for every rank-one projection $P$. 
	
	We further assume that $P\ne I$ is a projection of rank greater than one. Then  $\alpha P + \beta I$ has two points in the spectrum and by Lemma \ref{lem:1.8},	$\phi(\alpha P +\beta I)$ has two points in the spectrum, too. 
	The operator $\phi(\alpha P +\beta I)$ commutes with $\phi(R)$ for every rank-one operator  $R\le P$, i.e. $PR=RP=R$ and, $\phi(\alpha P +\beta I)$ commutes with  every rank-one $R\perp P$, Therefore, $\phi(\alpha P )\equiv g_P(\alpha)P$ for some function $g_P$ and similarly as above, $\phi(\alpha P +\beta I )\equiv g_P(\alpha)P$. Moreover, for every rank-one $R\le P$, by application of \eqref{eq:A-B} with $A=\alpha P$, $B=\alpha R$ we get 
	$$\phi(\alpha P)-\phi(\alpha R)\in \phi(\alpha(P-R))^{cc}=(P-R)^{cc}$$ 
	and so, $\phi(\alpha P)-\phi(\alpha R)\equiv \gamma (P-R)$. Observe that $\phi(\alpha P)\equiv \phi(\alpha R)+ \gamma (P-R)$. As $R$ is of rank one, $P$ is nontrivial of rank  greater than one, the operators $R$, $P-R$ and $I$ are linearly independent, which gives that $g_P=f_R$.
	
	{\sc Step 2.}  \textit{The function $f$ from Step 1 is linear: $f(a)=ca$ for every $a\in \R$ and a fixed nonzero} $c\in \R$.
	
	We will make use of operators $A$, $E$ and $F$ from Lemma \ref{lem:AEF}. As $A\equiv -2\lambda P$, $E\equiv -2aQ$ and $F\equiv -2R$, they are all primitive and so, $\phi(A)\equiv \frac {f(-2\lambda)}{-2\lambda}A$, $\phi(E)\equiv \frac{f(-2a)}{-2a}E$ and $\phi(F)\equiv \frac{f(-2)}{-2}F$. Since $A-E\com F$, $A-F \qcom E$ we have $\phi(A)-\phi(E)\com \phi(F)$ and   $\phi(A)-\phi(F) \com \phi(E)$. Applying Lemma \ref{lem:AEF} gives $\frac{f(-2\lambda)}{-2\lambda}=\frac{f(-2a)}{-2a}=\frac{f(-2)}{-2}$ for every nonzero $a$ which proves our claim.
	
	Without losing any generality, we further assume that $c=1$.
	
	
	{\sc Step 3.}
	\textit{For any positive integer $n$ and any set of linearly independent nontrivial projections $\{P_1, \ldots , P_n\}$  and $A\equiv \sum_{i=1}^{n} \alpha_i P_i$, $\alpha_i\ne 0$, $i=1,2,\dots,n$, we have
		\[ \phi(A) \equiv A.\]
	}
	We prove the assertion in the Step by induction. For $n=1$, Steps 1 and 2 prove the claim.
	So suppose that our assertion is true for  $n \in \mathbb{N}$. We will show that the statement holds also for $n + 1$. If $P_1$, $P_{n+1}$, $I$ are linearly dependent, then it is easy to check that $P_1+P_{n+1}=I$. In this case, the number of projections in the representation of $A$ can be reduced and hence, the induction hypothesis can immediately be applied. So, let $P_1,P_{n+1},I$ be linearly independent and compute
	
	\begin{equation}\label{eq:0010}
		\{\phi(\sum_{i=1}^{n+1} \alpha_i P_i)-\phi(\sum_{i=1}^{n} \alpha_i P_i)\}^{cc}=\{\phi(\alpha_{n+1}P_{n+1})\}^{cc}=\{P_{n+1}\}^{cc}
	\end{equation}
	and
	\begin{equation}\label{eq:0020}
		\{\phi(\sum_{i=1}^{n+1} \alpha_i P_i)-\phi(\sum_{i=2}^{n+1} \alpha_i P_i)\}^{cc}=\{\phi(\alpha_1P_1)\}^{cc}=\{P_1\}^{cc}.
	\end{equation}
	Since the second commutant of any non-trivial projection $P$ is of the form $\alpha P +\beta I$ for some scalars $ 0 \neq \alpha\in \mathbb{R}$ and $\beta\in \mathbb{C}$, from \eqref{eq:0010}, \eqref{eq:0020} and by the  induction hypothesis, we get 
	\begin{align}
		&\phi(\sum_{i=1}^{n+1} \alpha_i P_i) \equiv \sum_{i=1}^{n} \alpha_i P_i + \beta_{n+1}P_{n+1} \label{7}\\
		&\phi(\sum_{i=1}^{n+1} \alpha_i P_i)\equiv \sum_{i=2}^{n+1}\alpha_i P_i + \beta_1P_1. \label{8}
	\end{align}
	By  relations ($\ref{7}$) and ($\ref{8}$) we get $\alpha_1 = \beta_1$ and $\alpha_{n+1} = \beta_{n+1}$, due to assumed linear independency of $P_1,P_{n+1}$ and $I$.
	
	{\sc Step 4.}
	\textit{ $\phi(A) \equiv A$ for every linear operator $A$.}

	Since every self-adjoint operator can be represented as a real-linear combination of no more than 8  projections \cite[Theorem 3]{Pearcy}, we can assume that these projections are linearly independent. In the final-dimensional case, every operator can be represented as a real-linear combination of orthogonal projections.   Now, the proof of Theorem $\ref{th:4}$ is being completed by Step 3.
	
	\section{Proof of Theorem $\ref{th:5}$}
	Assume that $ \phi:\bh \mapsto \bh  $ is a surjective transformation satisfying \eqref{eq:1}.
	By Lemma \ref{lem:scalar} and surjectivity of $\phi$, we have that $\phi(A)$ is scalar if and only if $A$ is scalar. 
	We proceed in a series of steps.
	
	{\sc Step 1.} $\phi$ preserves quasi-commutativity in both directions.
	
	Assume $\phi(0)=\lambda_0 I$ for some $\lambda_0 \neq 0$ and take a pair of $A,B$, such that $A\qcom B$ and hence, $A-0\qcom B$ as well as $B-0 \qcom A$ implying that  $\phi(A)-\lambda_0 I\qcom \phi(B)$ and  $\phi(B)-\lambda_0 I\qcom \phi(A)$. If $\phi(A)$ and $\phi(B)$ do not commute, then $\phi(A)-\lambda_0 I\jcom \phi(B)$ and  $\phi(B)-\lambda_0 I\jcom \phi(A)$. By Lemma \ref{lem:4} we get $\phi(A)=\phi(B)$, a contradiction. This show that $\phi(A)\qcom \phi(B)$. Now suppose that $\phi(A)\qcom \phi(B)$, then clearly $\phi(A)-0\qcom \phi(B)$ as well as $\phi(B)-0\qcom \phi(A)$. It is possible to fix a nonzero $\mu$ such that $\phi(\mu I)=0$. So, $A-\mu I\qcom B$ and $B-\mu I\qcom A$. Then we get again that $A\com B$ and in turn, $A\qcom B$ as desired. 
	
	
	Now we consider the case $\phi(0)=0$. That $A\qcom B$ implies $\phi(A)\qcom \phi(B)$ is straightforward. Assuming that $\phi(A)\qcom \phi(B)$ gives that $\phi(A)-0\qcom \phi(B)$ and $\phi(B)-0\qcom \phi(A)$. We know that $\phi(X)=0$ implies that $X=\mu I$ for some real number $\mu$. If $\mu=0$ is the only possible one, there is nothing to do. So assume $\mu \neq 0$. 
	It follows that $A-\mu I\qcom B$ and $B-\mu I\qcom A$ providing that $A \com B$ by Lemma \ref{lem:4} and thus $A\qcom B$.
	
	The equality $(\phi(A)-\phi(B))^{\#}=\phi(A-B)^{\#}$ follows directly from surjectivity of $\phi$, preserving property \eqref{eq:1} and Step 1. Consequently,
	\begin{equation}\label{eq:500}
		(\phi(A)-\phi(B))^{\#\#}=\phi(A-B)^{\#\#}.
	\end{equation}	 
	for every $A, B\in \bh$. We easily obtain also that $\phi(A)=\phi(B)$ implies that $A-B\in \R I$ for every $A$, $B\in \bh$.  Moreover, $\phi(\mathcal{M}^{\#\#})=\phi(\mathcal{M})^{\#\#}$ for every $\mathcal{M}\subset \bh$.

	{\sc Step 2.} \textit{$A^\#=A^c$ if and only if $\phi(A)^\#=\phi(A)^c$ for any  $A\in\bh$.}
	
	If $A$ is scalar, there is nothing to do, so let $A$ be a nonscalar operator. Assuming $\phi(A)^\#=\phi(A)^c$ gives in particular that $\phi(A)^\#$ is a vector space. So, taking any $Y,Z\in \phi(A)^\#$ implies that $Y-Z \in \phi(A)^\#$. Passing to the preimages and choosing any pair of $Y_0$, $Z_0$ such that $\phi(Y_0)=Y$ and $\phi(Z_0)=Z$ gives first that $Y_0$, $Z_0\in A^\#$ and $Y_0-Z_0\in A^\#$. Since $\R I \subset  \phi(A)^c$, we can choose $Z_0=I$. Suppose that $Y_0 \notin A^c$. Then as well $Y_0 -I \notin A^c$. But now $AY_0+Y_0A=0$ and $A(Y_0-I)+(Y_0-I)A=0$, a contradiction. This yields $A^\#\subseteq A^c$, and the reverse inclusion is obvious.  The proof of the implication in the other direction is similar and even simpler.
	
	{\sc Step 3.} $\phi$ \textit{preserves all operators with two spectral points which do not add up to zero, in both directions. Moreover, $\phi$ preserves all primitive operators of this kind in both directions. }
	
	We apply Lemmas \ref{lem:1.81} and \ref{lem:primitive1} for all operators with two spectral points that do not add up to zero.
	
	Then we define a map $\psi$ on rank-one projections such that $\phi(P)=Q$ when $\phi(P)= \alpha Q + \beta I$. In the same way as in the proof of Theorem \ref{th:4}, we show that $\psi$ is bijective and preserve orthogonality. Therefore, there exists a unitary or antiunitary operator $U$ such that $\psi(R)=URU^\ast$ for every rank-one projection $R$. There is no loss of generality in assuming that $U=I$.  
	
	Now we can follow the proof of Theorem \ref{th:4} similarly as in Step 1. Let us denote by $\mathcal{K}$ a set of all two-spectral points operators whose spectral points add up to zero. 
	$$\mathcal{K}=\{\alpha(I-2P); P \text{ a nontrivial projection, }0\neq \alpha \in \R\}.$$
	
	{\sc Step 4.} 
	\textit{There exists a real-valued function $f$ such that for every $A=\alpha P + \beta I\notin \mathcal{K}$ we have  }  
	$$\phi(A) = f(\alpha)P +\mu_A I\notin \mathcal{K}$$
	where $\mu_A$ is a real constant, dependent on $A$.
	
	The proof of this step is analogous to the proof of Step 1 in the commutativity case for we use only the quasi-commutativity properties of operators $B$ having the property that $B^\#=B^c$.
	
	{\sc Step 5.} 
	\textit{For any  $K=I-2P\in \mathcal{K}$ and $0\neq \alpha \in \R$ we have $\phi(\alpha K)=\frac{-f(-2\alpha)}{2\alpha}K$.}
	
	We apply that $\alpha K- (-2\alpha P)=\alpha I$ and hence, by \eqref{eq:500} we obtain that 
	$$\phi(\alpha K)- \phi(-2\alpha P)\in (\phi(\alpha K)- \phi(-2\alpha P))^{\#\#} =\phi(\alpha I)^{\#\#} =\R I.$$
	Hence, $\phi(\alpha K)= f(-2\alpha) P +\gamma I$, where $f(-2\alpha)  +2\gamma =0 $ since otherwise we are in contradiction with Step 3. Then $\gamma =\frac{-f(-2\alpha)}{2}$ and then the claim easily follows.
	
	{\sc Step 6.} \textit{The function $f$ is linear. }
	
	We will use operators, $A$, $E$ and $F$ from Lemma \ref{lem:AEF}. Recall that $\lambda = \sqrt{1+a^2}$. Note that $A,E,F\in \mathcal{K}$. Applying Step 5, we have 
	\begin{align*}
		\phi(A) &=  \frac{-f(-2\lambda)}{2\lambda} A \\
		\phi(E) &= \frac{-f(-2a)}{2a}E  \\
		\phi(F) &= \frac{-f(-2)}{2}F.
	\end{align*}
	By Lemma  \ref{lem:AEF} we have  $\frac{-f(-2\lambda)}{2\lambda}=\frac{-f(-2a)}{2a}=\frac{-f(-2)}{2}$ which gives that $f(a)= c a$, $a\in \R$, with a fixed real constant $c$ which proves the assertion.
	
	We may and we do further assume that $c=1$. 
	
	{\sc Step 7.} 
	\textit{For any natural number $n\ge 1$ and any set of linearly independent nontrivial projections $\{P_1, \ldots , P_n\}$  and $A= \sum_{i=1}^{n} \alpha_i P_i$, $\alpha_i\ne 0$, $i=1,2,\dots,n$, we have
		\[ \phi(A) = A + g_A(A) I.\]
	}
	We again prove the claim by induction. There is nothing to do if $n=1$ by the argument in the preceding paragraph. So suppose that $n\ge 2$ and assume that the claim already holds for $n$.  When $n=2$, it is possible that $P_1$, $P_2$ and $I$ are linearly dependent. It is an easy exercise to see that in this case $P_1+P_2=I$ since $P_1$ and $P_2$ are linearly independent. 
	
	Now, by a similar argument as in the case of commutativity, we have that 
	$$\phi(A) -\phi(\sum_{j=1}^{n}\alpha_jP_j) \in \phi(\alpha_{n+1} P_{n+1})^{\#\#}=P_{n+1}^{cc}$$
	and 
	$$\phi(A) -\phi(\sum_{j=2}^{n+1}\alpha_jP_j) \in \phi(\alpha_{1} P_{1})^{\#\#}=P_{1}^{cc}.$$
	Comparing the expressions above, we get
	$$\alpha_{1} P_{1}+ \gamma P_{n+1}\equiv \phi(\alpha_{1} P_{1})+ \gamma P_{n+1}  \equiv \delta P_1 + \phi(\alpha_{n+1} P_{n+1}) \equiv \delta P_{1}+ \alpha_{n+1} P_{n+1}$$

	If $P_1$, $P_{n+1}$ and $I$ are linearly independent, we get $\alpha_1=\delta$ and $\alpha_{n+1}=\gamma$ and we are done.   If $n>2$, either $P_1$, $P_2$ and $I$ are linearly independent or,  $P_1$, $P_3$ and $I$ are linearly independent for otherwise we would have obtained $P_2=P_3$. So, in this case, we could have assumed in the beginning that $P_1$, $P_{n+1}$ and $I$ are linearly independent. In the case $n=2$, it is indeed possible that $P_1$, $P_2$ and $I$ are linearly dependent in which case $P_1+P_2=I$. But this situation is covered by Steps 4, 5, and 6.

	At the very end, applying that every operator in $\bh$ is a finite real-linear combination of projections \cite{Pearcy}, Step 7 closes the proof.

	\bigskip 
	
	\textbf{Aknowledgement}
	The second author is supported by the Slovenian Research and Innovation Agency (core research program P1-0306).

\end{document}